\documentclass{amsart}

\usepackage{amsfonts}
\usepackage{amsmath}
\usepackage{amssymb}
\usepackage{amsthm}
\usepackage{color}

\newtheorem{Thm}{Theorem}
\newtheorem{Lem}[Thm]{Lemma}
\newtheorem{Prop}[Thm]{Proposition}
\newtheorem{Cor}[Thm]{Corollary}

\theoremstyle{definition}
\newtheorem{Def}[Thm]{Definition}

\theoremstyle{remark}
\newtheorem{Rem}[Thm]{Remark}

\numberwithin{equation}{section}

\def\CC{{\mathbb{C}}}
\def\QQ{{\mathbb{Q}}}

\def\ZZ{{\mathbb{Z}}}
\def\NN{{\mathbb{N}}}
\def\PP{{\mathbb{P}}}

\def\PP{{\mathbb{P}}}

\def\bigOsoft{\tilde{\mathcal{O}}}

\def\bigO{\mathcal{O}}
\def\bigOsoft{\tilde{\mathcal{O}}}
\def\dv{{\textrm{div}}}

\title[Computation of Darboux polynomials and rational first integrals]{Computation of Darboux polynomials and rational first integrals with bounded degree\\ in polynomial time}
\date{\today}

\author[G.~Ch\`eze]{Guillaume Ch\`eze}
\address{Institut de Math\'ematiques de Toulouse\\
Universit\'e Paul Sabatier Toulouse 3 \\
MIP B\^at 1R3\\
31 062 TOULOUSE cedex 9, France}
\email{guillaume.cheze@math.univ-toulouse.fr}

\begin{document}
	
\begin{abstract}
In this paper we study planar polynomial differential systems of this form:
$$\dfrac{dX}{dt}=\dot{X}=A(X,Y), \,\, \dfrac{dY}{dt}=\dot{Y}=B(X,Y),$$
where $A,B \in \ZZ[X,Y]$ and $\deg A \leq d$,  $\deg B \leq d$, $\| A \|_{\infty} \leq \mathcal{H}$ and  $\| B \|_{\infty} \leq \mathcal{H}$.
A lot of properties of  planar polynomial differential systems are related to irreducible Darboux polynomials of the corresponding derivation: $D=A(X,Y)\partial_X + B(X,Y) \partial_Y$. Darboux polynomials are usually computed with the method of undetermined coefficients. With this method we have to solve a polynomial  system. 
We show that this approach can give rise to the computation of an exponential number of reducible Darboux polynomials. Here we show that the Lagutinskii-Pereira's algorithm  computes irreducible Darboux polynomials with degree smaller than $N$, with a polynomial number, relatively to $d$,  $\log(\mathcal{H})$ and $N$, binary operations. We also give a polynomial-time method to compute, if it exists, a rational first integral with bounded degree.
\end{abstract}	
	
\maketitle


 \bibliographystyle{alpha}
\bibliography{biblio_darboux}

\begin{thebibliography}{00}



\bibitem[AHS03]{AHS}
S.~Abhyankar, W.~Heinzer, and A.~Sathaye.
\newblock Translates of polynomials.
\newblock In {\em A tribute to C. S. Seshadri (Chennai, 2002)}, Trends Math.,
  pages 51--124. Birkh\"auser, Basel, 2003.

\bibitem[AKS07]{Sombra_Krick_Aven}
M.~Avenda{\~n}o, T.~Krick, and M.~Sombra.
\newblock Factoring bivariate sparse (lacunary) polynomials.
\newblock {\em J. Complexity}, 23:193--216, 2007.

\bibitem[Aut91]{Autonne}
L.~Autonne.
\newblock Sur la th{\'e}orie des {\'e}quations diff{\'e}rentielles du premier
  ordre et du premier degr{\'e}.
\newblock {\em Journal de l'{\'E}cole Polytechnique}, 61:35--122, 1891.

\bibitem[BC08]{Bus_Che}
L.~Bus\'e and G.~Ch\`eze.
\newblock On the total order of reducibility of a pencil of algebraic plane
  curves.
\newblock Preprint, 2008.

\bibitem[BCS97]{Burgetal}
P.~B\"{u}rgisser, M.~Clausen, and M.~Shokrollahi.
\newblock {\em Algebraic Complexity Theory}, volume 315 of {\em (Grundlehren
  der mathematischen Wissenschaften)}.
\newblock Springer, 1997.

\bibitem[Ber70]{Berlekamp}
E.~R. Berlekamp.
\newblock Factoring polynomials over large finite fields.
\newblock {\em Math. Comp.}, 24:713--735, 1970.

\bibitem[BLS{\etalchar{+}}04]{BLSSW}
A.~Bostan, G.~Lecerf, B.~Salvy, \'E. Schost, and B.~Wiebelt.
\newblock Complexity {I}ssues in {B}ivariate {P}olynomial {F}actorization.
\newblock In {\em Proceedings of ISSAC 2004}, pages 42--49. ACM, 2004.

\bibitem[Bod08]{Bodin}
A.~Bodin.
\newblock Reducibility of rational functions in several variables.
\newblock {\em Israel J. Math.}, 164:333--347, 2008.

\bibitem[BP94]{BiPa1994}
D.~Bini and V.~Pan.
\newblock {\em Polynomial and matrix computations. {V}ol. 1}.
\newblock Progress in Theoretical Computer Science. Birkh\"auser Boston Inc.,
  Boston, MA, 1994.
\newblock Fundamental algorithms.

\bibitem[BvHKS09]{Belabasetal}
K.~Belabas, M.~van Hoeij, J.~Kl{\"u}ners, and A.~Steel.
\newblock Factoring polynomials over global fields.
\newblock {\em J. Th. Nombres Bordeaux}, 21:15--39, 2009.

\bibitem[CG03]{Chavarriga_Grau}
J.~Chavarriga and M.~Grau.
\newblock Some open problems related to 16b {H}ilbert problem.
\newblock {\em Sci. Ser. A Math. Sci. (N.S.)}, 9:1--26, 2003.

\bibitem[CGGL03]{chavarriga_giac_gine_llibre}
J.~Chavarriga, H.~Giacomini, J.~Gin{\'e}, and J.~Llibre.
\newblock Darboux integrability and the inverse integrating factor.
\newblock {\em J. Differential Equations}, 194(1):116--139, 2003.

\bibitem[Chr94]{Christopher1994}
C.~Christopher.
\newblock Invariant algebraic curves and conditions for a centre.
\newblock {\em Proc. Roy. Soc. Edinburgh Sect. A}, 124(6):1209--1229, 1994.

\bibitem[Chr99]{Christopher}
C.~Christopher.
\newblock Liouvillian first integrals of second order polynomial differential
  equations.
\newblock {\em Electron. J. Differential Equations}, pages No. 49, 7 pp.
  (electronic), 1999.

\bibitem[CL07]{CL}
G.~Ch{\`e}ze and G.~Lecerf.
\newblock Lifting and recombination techniques for absolute factorization.
\newblock {\em J. Complexity}, 23(3):380--420, 2007.

\bibitem[CLP07]{Christopher_Llibre_Pereira}
C.~Christopher, J.~Llibre, and J.~Vit{\'o}rio Pereira.
\newblock Multiplicity of invariant algebraic curves in polynomial vector
  fields.
\newblock {\em Pacific J. Math.}, 229(1):63--117, 2007.

\bibitem[CMS06]{Coutinho_Schechter}
S.~C. Coutinho and L.~Menasch{\'e}~Schechter.
\newblock Algebraic solutions of holomorphic foliations: an algorithmic
  approach.
\newblock {\em J. Symbolic Comput.}, 41(5):603--618, 2006.

\bibitem[CMS09]{Coutinho}
S.~C. Coutinho and L.~Menasch{\'e}~Schechter.
\newblock Algebraic solutions of plane vector fields.
\newblock {\em J. Pure Appl. Algebra}, 213(1):144--153, 2009.

\bibitem[Dar78]{Darboux}
G.~Darboux.
\newblock Memoire sur les {\'e}quations diff'{\'e}rentielles du premier ordre
  et du premier degr{\'e}.
\newblock {\em Bull. Sci. Math.}, 32:60--96, 123--144, 151--200, 1878.

\bibitem[DDdM02a]{Duarte2}
L.~G.~S. Duarte, S.~E.~S. Duarte, and L.~A. C.~P. da~Mota.
\newblock Analysing the structure of the integrating factors for first-order
  ordinary differential equations with {L}iouvillian functions in the solution.
\newblock {\em J. Phys. A}, 35(4):1001--1006, 2002.

\bibitem[DDdM02b]{Duarte3}
L.~G.~S. Duarte, S.~E.~S. Duarte, and L.~A. C.~P. da~Mota.
\newblock A method to tackle first-order ordinary differential equations with
  {L}iouvillian functions in the solution.
\newblock {\em J. Phys. A}, 35(17):3899--3910, 2002.

\bibitem[DDdMS02]{Duarte1}
L.~G.~S. Duarte, S.~E.~S. Duarte, L.~A. C.~P. da~Mota, and J.~E.~F. Skea.
\newblock An extension of the {P}relle-{S}inger method and a {M}aple
  implementation.
\newblock {\em Comput. Phys. Comm.}, 144(1):46--62, 2002.

\bibitem[DLA06]{Dumortier_Llibre_Artes}
F.~Dumortier, J.~Llibre, and J.~C.~Art{\'e}s.
\newblock {\em Qualitative theory of planar differential systems}.
\newblock Universitext. Springer-Verlag, Berlin, 2006.

\bibitem[DLS98]{Lagutinskii}
V.~A.~Dobrovol{'}skii, N.V.~Lokot{'}, and
  J.-M.~Strelcyn.
\newblock Mikhail {N}ikolaevich {L}agutinskii (1871--1915): un
  math{\'e}maticien m{\'e}connu.
\newblock {\em Historia Math.}, 25(3):245--264, 1998.

\bibitem[DT89]{DT89}
R.~Dvornicich and C.~Traverso.
\newblock Newton symmetric functions and the arithmetic of algebraically closed
  fields.
\newblock In {\em Applied algebra, algebraic algorithms and error-correcting
  codes ({M}enorca, 1987)}, volume 356 of {\em Lecture Notes in Comput. Sci.},
  pages 216--224. Springer, Berlin, 1989.

\bibitem[FG10]{Ferragut_Giac}
A.~Ferragut and H.~Giacomini.
\newblock A new algorithm for finding rational first integrals of polynomial
  vector fields.
\newblock to appear in Qualitative Theory of Dynamical Systems, 2010.

\bibitem[Gao03]{Gao}
S.~Gao.
\newblock Factoring multivariate polynomials via partial differential
  equations.
\newblock {\em Math. Comp.}, 72(242):801--822 (electronic), 2003.

\bibitem[Gin07]{Gine_survey}
J.~Gin{\'e}.
\newblock On some open problems in planar differential systems and {H}ilbert's
  16th problem.
\newblock {\em Chaos Solitons Fractals}, 31(5):1118--1134, 2007.

\bibitem[GLV96]{Giac_Llibre_Viano}
H.~Giacomini, J.~Llibre, and M.~Viano.
\newblock On the nonexistence, existence and uniqueness of limit cycles.
\newblock {\em Nonlinearity}, 9(2):501--516, 1996.

\bibitem[Gor01]{Goriely}
A.~Goriely.
\newblock {\em Integrability and nonintegrability of dynamical systems},
  volume~19 of {\em Advanced Series in Nonlinear Dynamics}.
\newblock World Scientific Publishing Co. Inc., River Edge, NJ, 2001.

\bibitem[Hew91]{Hewitt}
C.~G. Hewitt.
\newblock Algebraic invariant curves in cosmological dynamical systems and
  exact solutions.
\newblock {\em Gen. Relativity Gravitation}, 23(12):1363--1383, 1991.

\bibitem[HS75]{Horowitz}
E.~Horowitz and S.~Sahni.
\newblock On computing the exact determinant of matrices with polynomial
  entries.
\newblock {\em J. ACM}, 22(1):38--50, 1975.

\bibitem[Jou79]{Joua_Pfaff}
J.-P.~Jouanolou.
\newblock {\em \'{E}quations de {P}faff alg\'ebriques}, volume 708 of {\em
  Lecture Notes in Mathematics}.
\newblock Springer, Berlin, 1979.

\bibitem[Kal85a]{Kal85}
E.~Kaltofen.
\newblock Fast parallel absolute irreducibility testing.
\newblock {\em J. Symbolic Comput.}, 1(1):57--67, 1985.

\bibitem[Kal85b]{Kaltof1}
E.~Kaltofen.
\newblock Polynomial-time reductions from multivariate to bi- and univariate
  integral polynomial factorization.
\newblock {\em SIAM J. Comput.}, 14(2):469--489, 1985.

\bibitem[Kal90]{K2}
E.~Kaltofen.
\newblock Polynomial factorization 1982--1986.
\newblock In {\em Computers in mathematics (Stanford, CA, 1986)}, volume 125 of
  {\em Lecture Notes in Pure and Appl. Math.}, pages 285--309. Dekker, New
  York, 1990.

\bibitem[Kal92]{Kaltofsurvey}
E.~Kaltofen.
\newblock Polynomial factorization 1987-1991.
\newblock In I.~Simon, editor, {\em Proc. LATIN '92}, volume 583, pages
  294--313. Springer-Verlag, 1992.

\bibitem[KLL88]{KLL}
R.~Kannan, A.~K. Lenstra, and L.~Lov{\'a}sz.
\newblock Polynomial factorization and nonrandomness of bits of algebraic and
  some transcendental numbers.
\newblock {\em Math. Comp.}, 50(181):235--250, 1988.

\bibitem[Lec06]{Lec2006}
G.~Lecerf.
\newblock Sharp precision in {H}ensel lifting for bivariate polynomial
  factorization.
\newblock {\em Math. Comp.}, 75(254):921--933 (electronic), 2006.

\bibitem[Len84]{LLL_multi}
A.~K. Lenstra.
\newblock Factoring multivariate integral polynomials.
\newblock {\em Th. Comp. Science}, 34:207--213, 1984.

\bibitem[LLL82]{LLL}
A.~K. Lenstra, H.~W. Lenstra, Jr., and L.~Lov{\'a}sz.
\newblock Factoring polynomials with rational coefficients.
\newblock {\em Math. Ann.}, 261(4):515--534, 1982.

\bibitem[Lor93]{Lo}
D.~Lorenzini.
\newblock Reducibility of polynomials in two variables.
\newblock {\em J. Algebra}, 156(1):65--75, 1993.

\bibitem[LV08]{LlibreValls}
J.~Llibre and C.~ Valls.
\newblock Darboux integrability and algebraic invariant surfaces for the
  {R}ikitake system.
\newblock {\em J. Math. Phys.}, 49(3):032702, 17, 2008.

\bibitem[LZ00]{LlibreZhang_Rikitake}
J.~Llibre and X.~Zhang.
\newblock Invariant algebraic surfaces of the {R}ikitake system.
\newblock {\em J. Phys. A}, 33(42):7613--7635, 2000.

\bibitem[Man93]{Man}
Y.-K.~Man.
\newblock Computing closed form solutions of first order {ODE}s using the
  {P}relle-{S}inger procedure.
\newblock {\em J. Symbolic Comput.}, 16(5):423--443, 1993.

\bibitem[MM97]{ManMac}
Y.-K~Man and M.~A.~H.~MacCallum.
\newblock A rational approach to the {P}relle-{S}inger algorithm.
\newblock {\em J. Symbolic Comput.}, 24(1):31--43, 1997.

\bibitem[MO04]{Ollagnier}
J.~Moulin~Ollagnier.
\newblock Algebraic closure of a rational function.
\newblock {\em Qual. Theory Dyn. Syst.}, 5(2):285--300, 2004.

\bibitem[Pai91]{Painleve}
P.~Painlev{\'e}.
\newblock M{\'e}moire sur les {\'e}quations diff{\'e}rentielles du premier
  ordre.
\newblock {\em Annales Scientifiques de l'{\'E}cole Normale Sup{\'e}rieure},
  8:9--59, 103--140, 201--226, 276--284, 1891.

\bibitem[Per01]{Pereira2001}
J.V.~Pereira.
\newblock Vector fields, invariant varieties and linear systems.
\newblock {\em Ann. Inst. Fourier (Grenoble)}, 51(5):1385--1405, 2001.

\bibitem[Poi91]{Poincare}
H.~Poincar{\'e}.
\newblock Sur l'int{\'e}gration alg{\'e}brique des {\'e}quations
  diff{\'e}rentielles du premier ordre et du premier degr{\'e}.
\newblock {\em Rend. Circ. Mat. Palermo}, 5:161--191, 1891.

\bibitem[PS83]{PrelleSinger}
M.~J. Prelle and M.~F. Singer.
\newblock Elementary first integrals of differential equations.
\newblock {\em Trans. Amer. Math. Soc.}, 279(1):215--229, 1983.

\bibitem[Rup86]{Ruppert}
W.M.~Ruppert.
\newblock Reduzibilit\"at {E}bener {K}urven.
\newblock {\em J. Reine Angew. Math.}, 369:167--191, 1986.

\bibitem[Rup99]{Ruppert2}
W.M.~Ruppert.
\newblock Reducibility of polynomials {$f(x,y)$} modulo {$p$}.
\newblock {\em J. Number Theory}, 77(1):62--70, 1999.

\bibitem[Sch84]{Schonhage}
A.~Sch{\"o}nhage.
\newblock Factorization of univariate integer polynomials by {D}iophantine
  approximation and an improved basis reduction algorithm.
\newblock In {\em Automata, languages and programming ({A}ntwerp, 1984)},
  volume 172 of {\em Lecture Notes in Comput. Sci.}, pages 436--447. Springer,
  Berlin, 1984.

\bibitem[Sch00]{Schinzel}
A.~Schinzel.
\newblock {\em Polynomials with special regard to reducibility}, volume~77 of
  {\em Encyclopedia of Mathematics and its Applications}.
\newblock Cambridge University Press, Cambridge, 2000.
\newblock With an appendix by Umberto Zannier.

\bibitem[Sch07]{Scheib}
P.~Scheiblechner.
\newblock On the complexity of deciding connectedness and computing {B}etti
  numbers of a complex algebraic variety.
\newblock {\em J. Complexity}, 23(3):359--379, 2007.

\bibitem[Sin92]{Singer}
M.F.~Singer.
\newblock Liouvillian first integrals of differential equations.
\newblock {\em Trans. Amer. Math. Soc.}, 333(2):673--688, 1992.

\bibitem[Tra85]{Trag}
B.~Trager.
\newblock {\em On the integration of algebraic functions}.
\newblock PhD thesis, M.I.T., 1985.

\bibitem[Val05]{Valls}
C.~Valls.
\newblock Rikitake system: analytic and {D}arbouxian integrals.
\newblock {\em Proc. Roy. Soc. Edinburgh Sect. A}, 135(6):1309--1326, 2005.

\bibitem[Vis93]{Vi}
A.~Vistoli.
\newblock The number of reducible hypersurfaces in a pencil.
\newblock {\em Invent. Math.}, 112(2):247--262, 1993.

\bibitem[vzGG03]{GG}
J.~von~zur Gathen and J.~Gerhard.
\newblock {\em Modern computer algebra}.
\newblock Cambridge University Press, Cambridge, second edition, 2003.

\bibitem[Wei]{Weim}
M.~Weimann.
\newblock A lifting and recombination algorithm for rational factorization of
  sparse polynomials.
\newblock {\em J. Complexity}, to appear.

\bibitem[Yap00]{Yap}
C.K.~Yap.
\newblock {\em Fundamental problems of algorithmic algebra}.
\newblock Oxford University Press, Inc., New York, NY, USA, 2000.


\end{thebibliography}
\section*{Introduction}
In this paper we study the following planar polynomial differential system:
$$\dfrac{dX}{dt}=\dot{X}=A(X,Y), \,\, \dfrac{dY}{dt}=\dot{Y}=B(X,Y),$$
where $A,B \in \ZZ[X,Y]$ and $\deg A \leq d$,  $\deg B \leq d$, $\| A \|_{\infty} \leq \mathcal{H}$ and  $\| B \|_{\infty} \leq \mathcal{H}$.  We associate  to this polynomial differential system the polynomial derivation $D=A(X,Y)\partial_X + B(X,Y)\partial_Y$. \\

A polynomial $f$ is said to be a \emph{Darboux polynomial},  if $D(f)=g.f$, where $g$ is a polynomial. A lot of properties of a polynomial differential system are related to \emph{irreducible} Darboux polynomials of the corresponding derivation $D$. Usually Darboux polynomials are computed with the method of undetermined coefficients. In other words, if we suppose that $\deg f \leq N$ then $D(f)=g.f$ gives a polynomial system in the unknown coefficients of $g$ and $f$. Then we can find $f$ and $g$ if we solve this system.  We will see that this strategy can give rise to the computation of an exponential number of reducible Darboux polynomials. \\
 In this paper we show that we can compute all the irreducible Darboux polynomials of degree smaller than $N$ with $\bigO\Big(\big(dN\log(\mathcal{H})\big)^{\bigO(1)}\Big)$ binary operations. Our strategy relies on the  factorization of the ecstatic curve as suggested by J.V.~Pereira in \cite{Pereira2001}.\\
This complexity result implies that, if we use the Prelle-Singer's strategy \cite{PrelleSinger}, then we can compute an integrating factor in polynomial-time. With this integrating factor we can then deduce an elementary  solution of the given differential equation. We also show that we can decide if $D$ has a rational first integral of degree $N$. Furthermore, we can  compute this rational first integral with $\bigO\Big( \big( dN \log (\mathcal{H})\big)^{\bigO(1)} \Big)$ binary operations.

\subsection*{Related results}
The computation of a first integral of a polynomial differential system is an old and classical problem. The situation is the following: we want to compute a function $\mathcal{F}$ such that the curves $\mathcal{F}(X,Y)=c$, where $c$ are constants, give  orbits of the differential system. Thus we want to find  a function $\mathcal{F}$ such that $D(\mathcal{F})=0$.\\
In 1878, G.~Darboux \cite{Darboux} gives a strategy to find first integrals. One of the tool developed by G.~Darboux is now called \emph{Darboux polynomials}.  There exist a lot of different names in the literature for Darboux polynomials, for example we can find: special integrals, eigenpolynomials, algebraic invariant curves, particular algebraic solutions or special polynomials.\\

Now recall briefly the history and the use of Darboux polynomials.\\
 G.~Darboux shows in \cite{Darboux} that if we have enough Darboux polynomials then there exists a rational first integral for $D$, i.e. $D(\mathcal{F})=0$ and $\mathcal{F} \in \CC(X,Y)$. More precisely, G.~Darboux shows that if we have $d(d+1)/2+2$ Darboux polynomials then the derivation has a rational  first integral which can be expressed by means of these polynomials. This work is improved by H.~Poincar\'e \cite{Poincare}, P.~Painlev\'e \cite{Painleve}, and L.~Autonne \cite{Autonne} at the end of the  XIX century.

Since 1979, a lot of new results appear. In his book \cite{Joua_Pfaff}, J.-P.~Jouanolou gives  a polynomial derivation with no Darboux polynomials. Then there exist polynomial differential systems  with no nontrivial rational first integral.

In \cite{PrelleSinger} M.~Prelle and M.~Singer give a structure theorem for polynomial differential systems with an elementary first integral. Roughly speaking an elementary function is a function which can be written in terms  of polynomials, logarithms, exponentials and algebraic extensions.  With this structure theorem the authors show how we can compute an elementary first integral if we know all the Darboux polynomials of $D$.

In \cite{Man} and \cite{ManMac} Y.-K Man  and M. MacCallum explain  how we can use  Prelle-Singer's method in practice. The bottleneck of this method is the computation of  Darboux polynomials. Indeed, the method of undetermined coefficients is used. With this method we have to solve a polynomial system, and  Y.-K Man in \cite{Man} explains how to solve this system with a Groebner basis. The resolution of this polynomial system is difficult because no particular structure of this system is known. Furthermore, this polynomial system can give an exponential number of reducible Darboux polynomials.  

In \cite{Singer} M.~Singer shows a structure theorem for polynomial differential systems with a Liouvillian first integral. Roughly speaking a Liouvillian function is a function which can be obtained ``by quadratures'' of elementary functions. C.~Christopher in \cite{Christopher} improves this structure theorem, and then  he suggests an algorithm to find Liouvillian first integrals.

In a series of papers \cite{Duarte1,Duarte2,Duarte3} the authors describe an algorithm and its implementation for the computation of Liouvillian first integrals. The strategy is in the same spirit as the one proposed by C.~Christopher.

Singer's theorem, Christopher's theorem, and Duarte-Duarte-da Mota's algorithm are based on Darboux polynomials.

In \cite{Christopher_Llibre_Pereira}  the authors define an algebraic, a geometric, an  integral, an infinitesimal and an holonomic  notion of multiplicity for an irreducible Darboux polynomial. They show that these notions are related. Furthermore, in order to define the algebraic multiplicity they introduce the ecstatic curve and then use some of its properties. We recall the definition and some properties of the ecstatic curve in Section \ref{Sect_ecstatic}.  It seems that this curve and some of its properties was already known by M.N. Lagutinskii (1871--1915), see \cite{Lagutinskii}, and was rediscovered by J.V. Pereira, see  \cite{Pereira2001}. Theorem 2 in \cite{Pereira2001} is the main tool of our algorithms.

In \cite{Coutinho,Coutinho_Schechter}, the authors give a Las Vegas strategy to decide if there exist Darboux polynomials for a given derivation.

Recently, in \cite{Ferragut_Giac} the authors propose an algorithm to compute a rational first integral without computing Darboux polynomials. Unfortunately, there are no complexity study for this algorithm.\\

Darboux polynomials are also used in the qualitative study of polynomial system. For example, the inverse integrating factor is a special Darboux polynomial, and, algebraic limit cycles are factors of the inverse integrating factor, see e.g. \cite{Giac_Llibre_Viano}.

We also mention that Darboux polynomials are used in Physics: In \cite{Hewitt} the author uses Darboux polynomials in order to find exact solutions to the Einstein field. In \cite{LlibreZhang_Rikitake,Valls,LlibreValls} Darboux polynomials are also used to study the Rikitake system, which is a simple model for the earth's magneto-hydrodynamic.\\

 In \cite{Chavarriga_Grau,Gine_survey} the reader can find some open questions and some relations between the computation of Darboux polynomials and Hilbert's 16th problem.\\

For other results, references or applications of Darboux polynomials the reader can consult for example \cite{Goriely,Dumortier_Llibre_Artes}.

\subsection*{Main results}
In this paper we show:
\begin{Thm}\label{Thm1}
Let $D=A(X,Y)\partial_X+ B(X,Y)\partial_Y$ be a polynomial derivation such that $A(X,Y), B(X,Y) \in \ZZ[X,Y]$, $\deg A \leq d$, $\deg B \leq d$, $\|A\|_{\infty} \leq \mathcal{H}$, $\|B\|_{\infty} \leq \mathcal{H}$ and $A$, $B$ are coprime.
\begin{enumerate}
\item We can decide if there exists a finite number of irreducible Darboux polynomials with degree smaller than $N$ in a deterministic way with \\ 
\mbox{$\bigO\Big( \big(dN\log(\mathcal{H})\big)^{\bigO(1)}\Big)$} binary operations.
\item If there exists a finite number of irreducible Darboux polynomials with degree smaller than $N$ then we can compute  all of them in a deterministic way with $\bigO\Big( \big(dN\log(\mathcal{H})\big)^{\bigO(1)}\Big)$ binary operations.
\end{enumerate}
\end{Thm}
To author's knowledge this is the first polynomial-time result on the computation of Darboux polynomials.\\

In this paper we suppose that $A$ and $B$ have integer coefficients and are coprime. This hypothesis is not restrictive. Indeed, if $A$ and $B$ have rational coefficients then we can always reduce  our study to the case of integral coefficients by clearing denominators.   Furthermore if $A$ and $B$ have a nontrivial greatest common divisor $R$ then we set $ds=Rdt$. This gives a derivation $D_2= A/R \partial_X+ B/R \partial_Y$. If $D_2$ has a first integral then $D$ has a first integral, thus the hypothesis ``$A$ and $B$ are coprime'' is not restrictive.\\
Furthermore in this paper we describe algorithms in the bivariate case in order to emphasize their role in the Prelle-Singer's algorithm. Nevertheless, Proposition \ref{Pereira_Prop} which is our main tool is also true in the multivariate case, i.e. when we consider a derivation $D=A_1\partial_{X_1}+\cdots+A_n\partial_{X_n}$ where $A_i \in \ZZ[X_1,\ldots,X_n]$. Then the Lagutinskii-Pereira's algorithm, see Section \ref{Section_comput_darboux}, is also correct in the multivariate case.\\

If the polynomial derivation $D$ has an infinite number of Darboux polynomials then by Darboux's theorem, $D$ has a rational first integral. Thus in this situation the problem is the computation of a rational first integral. We prove the following result:

\begin{Thm}\label{Thm2}
Let $D=A(X,Y)\partial_X+ B(X,Y)\partial_Y$ be a polynomial derivation such that $A(X,Y), B(X,Y) \in \ZZ[X,Y]$, $\deg A \leq d$, $\deg B \leq d$, $\|A\|_{\infty} \leq \mathcal{H}$, $\|B\|_{\infty} \leq \mathcal{H}$ and $A$, $B$ are coprime.
\begin{enumerate}
\item We can decide if there exists a rational first integral $\mathcal{F}$ of degree smaller than $N$ with $\bigO\Big( \big(dN\log(\mathcal{H})\big)^{\bigO(1)}\Big)$ binary operations.
\item If there exists a rational first integral with degree smaller than $N$ then we can compute it in a deterministic way with $\bigO\Big( \big(dN\log(\mathcal{H})\big)^{\bigO(1)}\Big)$ binary operations.
\end{enumerate}
\end{Thm}

To author's knowledge this is the first polynomial-time result on  the computation of rational first integrals.\\

\subsection*{Structure of this paper}
In Section \ref{Section_classic} we recall some classical results about Darboux polynomials, the spectrum of a rational function and the complexity of  bivariate factorization. In Section \ref{Section_undet} we show that the method of undetermined coefficients can give an exponential number of reducible Darboux polynomials. In Section \ref{Sect_ecstatic} we give the definition and some properties of the ecstatic curve. In Section \ref{Section_comput_darboux} we prove Theorem \ref{Thm1} and in Section \ref{Section_rat_first_int} we prove Theorem \ref{Thm2}. At last in Section \ref{Sect-openquestions} we ask two questions about complexity and  polynomial differential equations.

\subsection*{Notations}
Let $f(X,Y) =\sum_{i,j} f_{i,j}X^iY^j \in \ZZ[X,Y]$ be a polynomial.\\
$\|f\|_{\infty}=max_{i,j}|f_{i,j}|$ is the height of the polynomial $f$.\\
$\deg f$ is the total degree of the polynomial $f$. The degree of a reduced rational function $p/q$ is the maximum of $\deg p$ and $\deg q$.\\
The bit-size of a bivariate polynomial $f$ is $(\deg f)^2\log(\|f\|_{\infty})$.\\
$\partial_X$ (resp. $\partial_Y$) denotes the derivative relatively to the variable $X$ (resp. $Y$).\\
We set: $\dv(A,B)=\partial_X A + \partial_Y B$.


\section{Classical results}\label{Section_classic}
\subsection{Darboux polynomials}
\begin{Def}
Let $D=A(X,Y) \partial_X + B(X,Y)\partial_Y$ be the derivation associated to the planar polynomial differential system 
$$\dot{X}=A(X,Y), \quad \dot{Y}=B(X,Y),$$
where $A,B \in \ZZ[X,Y]$.\\
A polynomial $f \in \CC[X,Y]$ is said to be a Darboux polynomial  of $D$ if there exists a polynomial $g \in \CC[X,Y]$ such that $D(f)=g.f$.\\
The polynomial $g$ is called the cofactor of $f$.\\
If $\mathcal{F}$ is a function such that $D(\mathcal{F})=0$ then we said that $\mathcal{F}$ is a first integral of $D$.
\end{Def}

\begin{Prop}\label{propsomcof}
Let $f$ be a polynomial and let $f=f_1f_2$ be a factorization of $f$ where $f_1$ and $f_2$ are coprime. Then $f$ is a Darboux polynomial with cofactor $g$ if and only if $f_1$ and $f_2$ are Darboux polynomials with  cofactors $g_1$ and $g_2$. Furthermore $g=g_1+g_2$.
\end{Prop}
\begin{proof}
See for example Lemma 8.3 page 216 in \cite{Dumortier_Llibre_Artes}.
\end{proof}

Now we recall Darboux's Theorem.

\begin{Thm}[Darboux's Theorem]\label{darbouxthm}
Let $A,B \in \ZZ[X,Y]$ and let $D=A \partial_X+B\partial_Y$. If $f_1,\dots,f_m \in \CC[X,Y]$ are relatively prime irreducible Darboux polynomials for $i=1,\dots,m$, then either $m<d(d+1)/2+2$ where $d=\max(\deg A,\deg B)$ or there exist integers $n_i$ not all zero such that $D(w)=0$, where $w=\prod_{i=1}^mf_i^{n_i}$. In the latter case, if $f$ is any irreducible Darboux polynomial, then either there exists $\lambda, \mu$ in $\CC$, not both zero such that $f$ divides $\lambda\prod_{i \in I}f_i^{n_i}-\mu\prod_{j \in J}f_j^{-n_j}$ where $I=\{i\, | \, n_i \geq 0\}$ and $J=\{j \, |\, n_j <0\}$, or $G$ divides $\gcd(A,B)$.
\end{Thm}

\begin{proof}
See \cite{Darboux} or \cite{Singer}.
\end{proof}

\begin{Def}
A function $R$ is an integrating factor if $$D(R)=-\dv(A,B)R.$$
\end{Def}

\begin{Rem}
We remark that an integrating factor satisfies also one of these equivalent conditions:
$$\dv(RA,RB)=0, \,\, \partial_X(RA)=-\partial_Y(RB).$$
\end{Rem}
If we know an integrating factor then we can deduce a first integral. Indeed, if $R$ is an integrating factor then
\begin{equation}\label{int_prem}
 \mathcal{F}=\int RBdX-\int\Big( RA+\partial_Y \int RB dX \Big)dY
\end{equation}
is a first integral.
\subsection{Elementary solutions}
The following theorem  is due to  M.~Prelle and M.~Singer, see \cite{PrelleSinger}.
\begin{Thm}[Prelle-Singer]
If a polynomial differential system has an elementary first integral then there exists an integrating factor  which is a $K$th  root of a rational function.
\end{Thm}

This gives the following method to compute an elementary first integral. This method is called the Prelle-Singer's method. Here we follow the description given in \cite{Man}.\\

 \textbf{\textsf{Prelle-Singer's method}}\\
\textsf{Input:} $D=A(X,Y)\partial_X+B(X,Y)\partial_Y$ a polynomial derivation, and  $N$ an integer.\\
\textsf{Output:}  An elementary  first integral, if it exists,  constructed with Darboux polynomials of degree smaller than $N$.  \\
\begin{enumerate}
\item Set $n=1$.
\item \label{step2}Find all monic irreducible polynomials $f_i$ such that $\deg f_i\leq n$ and $f_i$ divides $D(f_i)$.
\item \label{step3}Let $D(f_i)=g_if_i$. Decide if there are constants $n_i$, not all zero, such that
$$\sum_{i=1}^m n_ig_i=0.$$
If such $n_i$ exist then $\prod_{i=1}^m f_i^{n_i}$ is a first integral.\\
If no such $n_i$ exist then go to the next step.
\item \label{step4}Decide if there are constants $n_i$, such that
$$\sum_{i=1}^m n_i g_i=-\dv(A,B).$$
If such $n_i$ exist then $\prod_{i=1}^m f_i^{n_i}$ is an integrating factor for the given differential equation. A first integral is given by the formula (\ref{int_prem}).\\
If no such $n_i$ exists then go to the next step.
\item Increase the value $n$ by 1. If $n$ is greater than $N$ then return failure otherwise repeat the whole procedure.
\end{enumerate}


If we want to compute an integrating factor, it is well known, see \cite{Man}, that step \ref{step2} is the most difficult step of the Prelle-Singer's method. Indeed, when we have all the irreducible Darboux polynomials we just have to solve linear systems (see Step  \ref{step3} and Step \ref{step4}) to deduce and integrating factor.

\begin{Rem}
In Prelle-Singer's method the user must give a bound $N$. Nowadays we cannot remove this  input. Indeed, we do not know a bound on the maximal degree of irreducible Darboux polynomials of a given derivation. This is an open question and appears in \cite{Poincare,PrelleSinger}. The following example shows that we cannot get a bound  in term of the degree  of $D$ only. The bound must also take into account the coefficients of $D$.\\
\emph{Example:} The derivation $D= (n+1)X\partial_X + nY\partial_Y$
has  $X^n-Y^{n+1}$ as Darboux polynomial.
\end{Rem}

\subsection{Spectrum of a rational function}
In this subsection we recall the definition and a property of the spectrum of a rational function. This notion is used in Section \ref{Section_rat_first_int}.

\begin{Def}
A rational function  $f(X,Y) \in \QQ(X,Y)$  is said to be composite if it can be written $f=u \circ h$ where $h(X,Y) \in \QQ(X,Y)$ and $u \in\QQ(T)$ such that $\deg(u) \geq 2$. Otherwise $f$ is said to be non-composite.
\end{Def}

\begin{Def}
Let $f=p/q \in \QQ(X,Y)$ be a reduced rational function of degree $d$. The set 
\begin{eqnarray*}
\sigma(p,q) =\{(\lambda:\mu) \in \PP^1_{\CC} &\mid &\lambda p+ \mu q\textrm{ is reducible in } \CC[X,Y], \\
&&\textrm{ or } \deg (\lambda p + \mu q ) < d \, \}
\end{eqnarray*} 
is the spectrum of $f=p/q$. We recall that  a polynomial reducible in $\CC[X,Y]$ is said to be absolutely reducible.\\
\end{Def}

The spectrum $\sigma(p,q)$ is finite if and only if $p/q$ is non-composite and if and only if the pencil of algebraic  curves $\lambda p + \mu q=0$, has an irreducible general element (see for instance \cite[Chapitre 2, Th\'eor\`eme 3.4.6]{Joua_Pfaff} or \cite[Theorem 2.2]{Bodin} for detailed proofs).\\
 To author's knowledge, the first effective result on the spectrum has been given by Poincar\'e \cite{Poincare}. In this paper,  Poincar\'e gives a relation between the number of saddles of a polynomial vector field and the spectrum. He also shows that $ |\sigma(p,q)| \leq (2d-1)^2+2d+2$.
This bound was improved only recently by Ruppert \cite{Ruppert} who proved the following result:
\begin{Prop}\label{bornespectre}
If $p/q \in \QQ(X,Y)$ is a reduced non-composite rational function of degree $d$ then
$|\sigma(p,q)|\leq d^2-1$.
\end{Prop}
 This result was obtained as a byproduct of a very interesting technique developed to decide the reducibility of an algebraic plane curve.\\
Several papers improve this result,  see e.g. \cite{Lo,Vi,AHS,Bodin, Bus_Che}.

\begin{Rem}
If $p/q$ is a reduced non-composite rational first integral of $D$ then $p+\lambda q$, where $\lambda \in \CC$, are Darboux polynomials. Then $D$ has  infinitely many irreducible Darboux polynomials because the spectrum $\sigma(p,q)$ is finite.
\end{Rem}

\subsection{Complexity results}\label{Complex_result}
In this paper we consider the dense representation of polynomials in the usual monomial basis. We recall that we can factorize in a deterministic way an integer univariate polynomial $f$ with $\bigO\Big( \big(d \log(\mathcal{H})\big)^{\bigO(1)} \Big)$ binary operations,  where $d$ is the degree of $f$ and $\mathcal{H}$ its height, see e.g. \cite{Schonhage}. An algorithm with this kind of complexity is called a polynomial-time algorithm because its complexity is bounded by a polynomial in the size of the input. We recall that the bit-size of an univariate polynomial $f$ of degree $d$ and height $\mathcal{H}$ is $d \log (\mathcal{H})$.\\
The first polynomial-time algorithm for the factorization of univariate polynomials is due to Lenstra, Lenstra and Lovasz, see \cite{LLL}. However this algorithm, called LLL,  is probabilistic because it uses Berlekamp's algorithm which is probabilistic,  see \cite{Berlekamp}. Nevertheless, there exist deterministic polynomial-time algorithms, see e.g. \cite{Schonhage,KLL}. In \cite{Schonhage,KLL} the strategy is numerical. Instead of computing a modular factorization as in \cite{LLL}, the authors propose to compute a complex root with a sufficiently high precision.\\

In \cite{Kaltof1} the author shows that we can reduce bivariate factorization to univariate factorization. Then we get a deterministic polynomial-time algorithm for the factorization of integer bivariate polynomials. Another polynomial-time algorithm is proposed in \cite{LLL_multi}, the authors extend to the multivariate case the LLL algorithm.\\
Few time later, several papers see \cite{Kal85,DT89,Trag},   show that we can also compute the absolute factorization (i.e. the factorization in $\overline{\QQ}[X,Y]$, where $\overline{\QQ}$ is the algebraic closure of $\QQ$) of integer bivariate polynomials with a deterministic polynomial-time algorithm. Thus there exist deterministic algorithms which perform the absolute factorization of a bivariate polynomial $f(X,Y) \in \QQ[X,Y]$ of degree $d$ and height $\mathcal{H}$ with  $\bigO\Big( \big(d \log(\mathcal{H})\big)^{\bigO(1)} \Big)$  binary operations.\\
For the history of factorization's algorithm the reader can consult \cite{K2,Kaltofsurvey}. For more recent results about complexity and bivariate polynomials factorization see  \cite{Gao,BLSSW,Lec2006,Belabasetal,CL,Sombra_Krick_Aven,Weim}.\\

We also mention here that we can solve a linear system, and compute a determinant, with coefficients in $\ZZ$ or in $\ZZ[X,Y]$, in polynomial-time, see e.g. \cite{Yap,Burgetal,Horowitz}. This will be useful in our complexity analysis.

\section{The method of undetermined coefficients}\label{Section_undet}
Usually, Darboux polynomials are computed with the method of undetermined coefficients. In other words, we write: $D(f)=g.f$ with $\deg g \leq d$, where $d$ is the degree of $D$, and $\deg f \leq N$. This gives a polynomial system in the unknown coefficients of $g$ and $f$. This system has $\bigO(N^2+d^2)$ unknowns. In this section, we show that the method of undetermined coefficients can give  an exponential numbers of \emph{reducible} Darboux polynomials. We recall that only \emph{irreducible} Darboux polynomials are useful because the product of Darboux polynomials is a Darboux polynomial, see Proposition \ref{propsomcof}. 
\begin{Lem}\label{lem_exemple_expo}
If we consider the following derivation:
$$D= (\partial_Y \mathcal{F}) \partial_X - (\partial_X \mathcal{F}) \partial_Y, \textrm{ with } \mathcal{F}(X,Y)=Y \prod_{i=1}^{d-1} (X+i) +X, $$
then there are at least $2^{d-1}+1$ Darboux polynomials with degree smaller than $d$.
\end{Lem}
\begin{proof}
In this situation $X+i$, $i=1,\dots,d-1$, are irreducible Darboux polynomials. Thus 
$$\prod_{I} (X+i), \textrm{ where } I \textrm{ is a subset of } \{1, \dots, d-1\}$$
 is a Darboux polynomial. We have $2^{d-1}$ such Darboux polynomials and these polynomials have a degree smaller than $d-1$. Furthermore $\mathcal{F}$ is an irreducible Darboux polynomial and $\deg \mathcal{F}=d$. This gives the desired result.
\end{proof}

\begin{Rem}
In Lemma \ref{lem_exemple_expo}, we give a derivation with a rational first integral. It would be interesting to have the same kind of result with a derivation with no rational first integral.\\

In the situation of Lemma \ref{lem_exemple_expo}, $\mathcal{F}$  is an irreducible Darboux polynomial of degree $d$. Thus if we want to find all the irreducible Darboux polynomials, then we have to consider  polynomials $f$ with degree smaller or equal to $d$ in the polynomial system $D(f)=g.f$. Now, Lemma \ref{lem_exemple_expo} implies that 
\emph{we have at least $2^{d-1}+1$ solutions for the system  $D(f)=g.f$ with $\deg f\leq d$}.\\
Then we can conclude:  \emph{in the worst case the method of undetermined coefficients gives an exponential number, in $d$,  of reducible Darboux polynomials.}
\end{Rem}

The problem comes from  \emph{reducible} Darboux polynomials. The recombination of irreducible Darboux polynomials gives an exponential number of Darboux polynomials. We can ``avoid'' this problem. Indeed, we can add to the system $D(f)=g.f$, forms  $\varphi_{i,N}(f)$, $i=1,\dots,T$, such that $\varphi_{i,N}(f)=0$, for all $i=1, \dots, T$ if and only if $f$ is an absolutely reducible polynomials with degree $N$.
Such forms exist and they are called Noether's forms, see \cite{Ruppert,Schinzel}. Unfortunately, to author's knowledge, the number $T$ of these forms is exponential in $N$. More precisely, the number of Noether's forms is equal to $\binom{2N^2-3N+1}{N^2-1}$. With Stirling's formula we can show,  when $N$ tends to infinity:   
$$  T=\binom{2N^2-3N+1}{N^2-1} \sim \dfrac{2^{2N^2-3N+1}}{e^{9/4}\sqrt{\pi} N}.$$
Thus in this case we have an exponential number of equations, and then we get  a method with an exponential complexity in $N$.\\

We remark that there exists a strategy to compute the leading term of Darboux polynomials, see e.g.\cite{Christopher1994}. This strategy gives a net gain in practical examples, see \cite{ManMac}. However, this strategy do not detect the leading term of an irreducible Darboux polynomial. Then, even if we add this strategy to the method of undetermined coefficient, we still get an exponential number of Darboux polynomials.\\

The exponential complexity is related to the recombination of irreducible Darboux polynomials. Exponential complexity due to recombinations appears also when we study factorization algorithms. However, we can factorize polynomials in polynomial-time as mentioned in Section \ref{Complex_result}. In the following we  show that we can reduce the computation of irreducible Darboux polynomials to  the factorization of a bivariate polynomial. Then, we will deduce a polynomial-time algorithm for the computation of irreducible Darboux polynomials.

\section{The ecstatic curve}\label{Sect_ecstatic}
\begin{Def}
Let $D$ be a polynomial derivation, the $N$th ecstatic curve of $D$, $\mathcal{E}_{\mathcal{B},N}(D)$, is given by the polynomial
$$\det 
\left(
\begin{array}{cccc}
v_1 & v_2& \cdots& v_l\\
D(v_1)& D(v_2)& \cdots& D(v_l)\\
\vdots&\vdots& \cdots&\vdots\\
D^{l-1}(v_1)& D^{l-1}(v_2)& \cdots&D^{l-1}(v_l)
\end{array}
\right),
$$
where $\mathcal{B}=\{v_1,v_2,\dots,v_l\}$ is a basis of $\CC[X,Y]_{\leq N}$, the $\CC$-vector space of polynomials in $\CC[X,Y]$ of degree at most $N$, $l=(N+1)(N+2)/2$, and $D^{k}(v_i)=D\big(D^{k-1}(v_i)\big)$.\\
When $\mathcal{B}$ is the monomial basis, we denote by $\mathcal{E}_N(D)$ the ecstatic curve.
\end{Def}

Now, we remark that the ecstatic curve is independent of the chosen basis of $\CC[X,Y]_{\leq N}$ up to a multiplicative constant.

\begin{Prop}\label{chg_base}
Let $\mathcal{B}$ be a basis of $\CC[X,Y]_{\leq N}$. We have 
$$\mathcal{E}_{\mathcal{B},N}(D)=c.\mathcal{E}_N(D),$$
where $c \in \CC^{*}$.

\end{Prop}
\begin{proof}
We are going to show that $\mathcal{E}_{\mathcal{B},N}(D)$ is the determinant of an endomorphism. Consider
\begin{eqnarray*}
\mathcal{D}: \CC(X,Y)[U,V]_{\leq N}& \longrightarrow& \CC(X,Y)[U,V]_{\leq N}\\
U^iV^j& \longmapsto& \sum_{0 \leq k+l \leq N} D^{\textrm{cantor}(k,l)}(X^iY^j)U^kV^l
\end{eqnarray*}
where $U$ and $V$ are new independent variables, $\CC(X,Y)[U,V]_{\leq N}$ is the $\CC(X,Y)$ vector space of polynomials with coefficients in $\CC(X,Y)$, and with a degree relatively to  $U$ and $V$ at most $N$. Furthermore $\textrm{cantor}$ is the following map:
\begin{eqnarray*}
\textrm{cantor}: \NN^2& \longrightarrow & \NN\\
(k,l)& \longmapsto& \dfrac{(k+l)^2+3k+l}{2}.
\end{eqnarray*}

This application maps $(0,0)$ to 0, $(0,1)$ to 1,  $(1,0)$ to 2, $(0,2)$ to 3\dots That is to say,  the monomials $U^kV^l$ are ordered with a graded lexicographic order.

$\mathcal{D}$ is defined on a basis thus $\mathcal{D}$ is a well defined  endomorphism. $\mathcal{E}_N(D)$ is the determinant of this endomorphism written in the monomial basis.

Now we remark that: 
$$\mathcal{D}\Big(\sum_{0\leq i+j\leq N} p_{i,j} U^iV^j\Big)=\sum_{0\leq k+l \leq N}\Big( D^{\textrm{cantor}(k,l)}(\sum_{ 0 \leq i+j\leq N} p_{i,j}X^iY^j) \Big) U^kV^l,$$
where $p_{i,j} \in \CC$.\\
Each element of the basis $\mathcal{B}$ can be written $\sum_{i,j}p_{i,j}X^iY^j$ with $p_{i,j} \in \CC$.  We can also consider a basis $\mathcal{B}'$ of  the $\CC(X,Y)$ vector-space $\CC(X,Y)[U,V]_{\leq N}$, such that each element is written $\sum_{i,j}p_{i,j}U^iV^j$ with $p_{i,j} \in \CC$. Thus $\mathcal{E}_{\mathcal{B},N}(D)$ is the determinant of $\mathcal{D}$ written with $\mathcal{B}'$ in the domain, and  with the monomial basis in the target space. As $p_{i,j} \in \CC$, this gives the desired result.
\end{proof}

The following proposition is due to J.-V. Pereira, see \cite[Proposition 1]{Pereira2001}. It is the key point of our algorithm: it shows that the computation of Darboux polynomials can be reduced to the factorization of $\mathcal{E}_N(D)$. We give a proof in order to ease the readability of the paper.

\begin{Prop}\label{Pereira_Prop}
Every Darboux polynomial, relatively to $D$, of degree smaller than $N$ is a factor of $\mathcal{E}_N(D)$.
\end{Prop}

\begin{proof}
Let $F \in \CC[X,Y]_{\leq N}$ be  a Darboux polynomial.
By Proposition  \ref{chg_base} we can choose a basis $\mathcal{B}$ where $v_1=F$.\\
Furthermore, we have:
\begin{eqnarray*}
D(F)&=&g_1 F,\\
D^2(F)&=&D(g_1 F)=\big(g_1^2 +D(g_1)\big)F=g_2 F,\\
&\vdots&\\
D^{l-1}(F)&=&g_{l-1} F,
\end{eqnarray*}
where $g_1$, $g_2$, \dots, $g_{l-1}$ are polynomials.\\
Thus $F$ is a factor of $\mathcal{E}_{\mathcal{B},N}(D)$ and this concludes the proof.
\end{proof}

\begin{Rem}
The converse is false. Indeed, consider the derivation $D=-2X^2\partial_X+(1-4XY)\partial_Y$. Then $\mathcal{E}_1(D)=YX^4$ but $Y$ is not a Darboux polynomial.
\end{Rem}

We know by Darboux's Theorem, see Theorem \ref{darbouxthm}, that if a derivation has a rational first integral then there are infinitely many irreducible Darboux polynomials. Thus if $D$ has a rational first integral then by Proposition \ref{Pereira_Prop}, $\mathcal{E}_N(D)$ has infinitely many irreducible factors. This gives $\mathcal{E}_N(D)=0$, and also $\mathcal{E}_M(D)=0$ for $M$  bigger than $N$. The following proposition says that the converse is also true. This proposition will be useful in Section \ref{Section_rat_first_int} when we will study the computation of rational first integrals.

\begin{Prop}\label{ecstatic_nulle}
We have $\mathcal{E}_N(D)=0$ and $\mathcal{E}_{N-1}(D) \neq 0$ if and only if $D$ admits a rational first integral of exact degree $N$.
\end{Prop}

\begin{proof}
See \cite[Theorem 1]{Pereira2001}.
\end{proof}

 In our complexity study we will need to know the bit-size of $\mathcal{E}_N(D)$. We recall that the bit-size of $\mathcal{E}_N(D)$ is: $\big(\deg(\mathcal{E}_N(D)\big)^2 \log\big( \| \mathcal{E}_N(d)\|_{\infty} \big)$. Thus in the following we are going to compute the degree and the height of $\mathcal{E}_{N}(D)$.

\begin{Prop}\label{Deg_extactic}
Let $D=A(X,Y)\partial_X+B(X,Y)\partial_Y$ be a polynomial derivation, where $\deg A, \deg B \leq d$. The degree of $\mathcal{E}_N(D)$ is at most $N.l+(d-1).(l-1).l/2$, where $l=(N+1)(N+2)/2$.
\end{Prop}

\begin{proof}
By definition of a determinant we have:
$$\deg \mathcal{E}_N(D) \leq \sum_{k=0}^{l-1} \deg\big(D^k(v_i)\big),$$
where $D^0(v_i)=v_i$.\\
A straightforward computation gives $\deg D^k(v_i) \leq k(d-1)+N$. Then 
\begin{eqnarray*}
\deg \mathcal{E}_N(D) &\leq& \sum_{k=0}^{l-1} \big(k(d-1)+N\big),\\
&\leq & N.l+(d-1)\sum_{k=0}^{l-1}k,\\
&\leq&N.l+(d-1).(l-1).l/2.
\end{eqnarray*}
This gives the desired result.
\end{proof}

\begin{Cor}\label{bigodegecstatic}
Under the hypothesis of Proposition \ref{Deg_extactic}, we have $\deg \mathcal{E}_N(D)$ belongs to $\bigO(dN^4)$.
\end{Cor}

\begin{Prop}\label{height_ecstatic}
The height $\|\mathcal{E}_N(D)\|_{\infty}$ satisfies
$$\| \mathcal{E}_N(D)\|_{\infty} \leq \Big(2l\mathcal{H}\big(l(d-1)+N\big)^{3}\Big)^{l(l-1)/2}
.$$
\end{Prop}

\begin{proof}
First we recall that if $f_1$ and $f_2$ are two polynomials with total degree smaller than $d$ then we have
$$\|f_1 .f_2 \|_{\infty} \leq (d+1)^2 \|f_1\|_{\infty} \|f_2\|_{\infty}.$$
This gives for $f \in \CC[X,Y]_{\leq N}$,
$$\| D(f)\|_{\infty} \leq 2\mathcal{H}N^3\|f\|_{\infty}.$$
By induction, using $(N+k(d-1)+1)^2 \leq (N+k(d-1))^3$, we get 
$$\|D^i(f)\|_{\infty} \leq 2^i \mathcal{H}^i\Big(   \prod_{k=0}^{i-1} \big(k(d-1)+N\big)^3\Big) \|f\|_{\infty}, \textrm{ where } i \geq 1.$$
By definition of a determinant,  we have: 
$$\|\mathcal{E}_N(D)\|_{\infty} \leq \sum_{\sigma \in \mathfrak{S}_l} \prod_{i=0}^{l-1}\|D^{i}(v_{\sigma(i+1)})\|_{\infty},$$
where  $\{v_1,\dots,v_l\}$ is the monomial basis of $\CC[X,Y]_{\leq N}$.
We get then:
\begin{eqnarray*}
\| \mathcal{E}_N(D)\|_{\infty} & \leq & l! \prod_{i=1}^{l-1}  2^i \mathcal{H}^i\Big(   \prod_{k=0}^{i-1} \big(k(d-1)+N\big)^3\Big)\\
&\leq & l! \prod_{i=1}^{l-1}  2^i \mathcal{H}^i\Big(   \big(l(d-1)+N\big)^{3i}\Big)\\
&\leq &\Big(2l\mathcal{H}\big(l(d-1)+N\big)^{3}\Big)^{l(l-1)/2}.
\end{eqnarray*}
\end{proof}

\begin{Cor}\label{Cor-size}
Under the hypothesis of Proposition \ref{height_ecstatic}, the bit-size of $\mathcal{E}_N(D)$ belongs to $\bigO  \Big( \big(dN\log(\mathcal{H})\big)^{\bigO(1)}\Big)$.
\end{Cor}

\begin{proof}
By Corollary \ref{bigodegecstatic} and Proposition \ref{height_ecstatic} the bit-size of $\mathcal{E}_N(D)$ is bounded by 
$$(dN^4)^2.l.(l-1).\log  \Big(2l\mathcal{H}\big(l(d-1)+N\big)^{3}\Big) .$$
This gives the desired result.
\end{proof}

\section{$D$ has a finite number of irreducible Darboux polynomials}\label{Section_comput_darboux}
If $D$ has a finite number of irreducible Darboux polynomials then the ecstatic curve $\mathcal{E}_N(D)$ is non-zero by Proposition \ref{ecstatic_nulle}. It follows that we can compute irreducible Darboux polynomials with a bivariate factorization algorithm thanks to Proposition \ref{Pereira_Prop}. It seems that Proposition \ref{Pereira_Prop} has been proved by M.N.~Lagutinskii and rediscovered by J.V.~Pereira, see \cite[Theorem 2]{Pereira2001}. Then we call ``Lagutinskii-Pereira's algorithm'' the following algorithm.\\

\textbf{\textsf{Lagutinskii-Pereira's algorithm}}\\
\textsf{Input:} A polynomial derivation $D=A(X,Y)\partial_X+B(X,Y)\partial_Y$, and $N$ an integer.\\
\textsf{Output:}  The finite set $S$ of all the absolute irreducible Darboux polynomials with degree smaller than $N$ or ``There exists an infinite number of irreducible Darboux polynomials''.\\
\begin{enumerate}
\item $S=\{\}$.
\item Compute $\mathcal{E}_N(D)$. 
\item If $\mathcal{E}_N(D)=0$ then Return  ``There exists an infinite number of irreducible Darboux polynomials'' else go to step \ref{stepcompdarboux}, end If.
\item \label{stepcompdarboux} Compute the set ${f_1, \dots,f_m}$ of all  absolutely irreducible factors of $\mathcal{E}_N(D)$  with degree smaller than $N$ .
\item For  $i:=1,\dots,m$ do: If $\gcd(f_i,D(f_i))=f_i$ then add $f_i$ to $S$, end If, end For.
\item Return $S$.
\end{enumerate}

\begin{Prop}
The Lagutinskii-Pereira's algorithm is correct.
\end{Prop}

\begin{proof}
This is a straightforward consequence of Proposition \ref{Pereira_Prop}.
\end{proof}

Now we can prove Theorem \ref{Thm1}.
\begin{proof}
The Lagutinskii-Pereira's algorithm is correct and works with the claimed complexity. Indeed, we can compute $\mathcal{E}_N(D)$ in polynomial-time because it is a determinant. Furthermore, by Corollary \ref{Cor-size} we know that the bit-size of $\mathcal{E}_N(D)$ belongs to $\bigO \Big( \big(dN\log(\mathcal{H})\big)^{\bigO(1)}\Big)$. Then we can compute the absolutely irreducible factors of $\mathcal{E}_N(D)$ with $\bigO\Big( \big(dN\log(\mathcal{H})\big)^{\bigO(1)}\Big)$ binary  operations, see Section \ref{Complex_result}. As gcd computations can also be  performed in polynomial-time we obtain the desired result.
\end{proof}

We also deduce:
\begin{Cor}
Under conditions of Theorem \ref{Thm1}, if there exists an integrating factor $R$ such that $R=\prod_i f_i^{n_i}$, where $f_i$ are Darboux polynomials with degree smaller than $N$, then we can compute $R$ with  $\bigO\Big( \big(dN\log(\mathcal{H})\big)^{\bigO(1)}\Big)$ binary  operations.
 \end{Cor}

\begin{proof}
We compute Darboux polynomials with the Lagutinskii-Pereira's algorithm and then we solve a linear system as in Step \ref{step4} of the Prelle-Singer's method, see Section \ref{Section_classic}.
\end{proof}

This corollary implies that we can compute an integrating factor corresponding to an elementary first integral with the Prelle-Singer's method with  $\bigO\Big( \big(dN\log(\mathcal{H})\big)^{\bigO(1)}\Big)$ binary  operations.

\section{$D$ has a rational first integral}\label{Section_rat_first_int}
We have seen in Section \ref{Sect_ecstatic} that is easy to test if a derivation has a rational first integral with degree smaller than $N$. Indeed, we just have to compute the ecstatic curve $\mathcal{E}_N(D)$ and  check if it is the zero polynomial. In this section, we show how we can compute this rational first integral.
\subsection{Computation of a Darboux polynomials with degree $N$}
 In this section we suppose that $D$ has a rational first integral $p/q$ with degree $N$. Then $p$ and $q$ are Darboux polynomials with degree $N$. By Proposition \ref{ecstatic_nulle}, $\mathcal{E}_N(D)=0$. Thus we cannot compute $p$ and $q$ as factors of $\mathcal{E}_N(D)$. The strategy is then the following: compute one Darboux polynomial with degree $N$, compute its cofactor, and then deduce $p$ and $q$.\\
Now we explain how we can compute a Darboux polynomial with degree $N$.
\begin{Def}
Let $D$ be a polynomial derivation, $\mathcal{E}_{N,0}(D)$, is  the polynomial $\mathcal{E}_{\mathcal{B}^0,N}(D)$ where $\mathcal{B}^0$ is the monomial basis of  the $\CC$-vector space of polynomials in $\CC[X,Y]$ of degree at most $N$ with constant term equal to zero.
\end{Def}
We have the following property.

\begin{Prop}\label{Prop_ecst_moins_zero}
Let $p/q$ be  a non-composite rational first integral of $D$, such that $\deg(p/q)=N$, and $p(0,0)q(0,0) \neq 0$.
\begin{enumerate}
\item We have $\mathcal{E}_{N,0}(D) \neq 0$ in $\QQ[X,Y]$.
\item If we set $(\lambda_0,\mu_0)=\big(-q(0,0),p(0,0)\big)$ then $\lambda_0 p + \mu_0 q$ is a factor of $\mathcal{E}_{N,0}(D)$.
\end{enumerate}
\end{Prop}

\begin{proof}
This proof follows very closely the proof of Theorem 5.3 and Proposition 5.2 in \cite{Christopher_Llibre_Pereira}.\\
First, we prove that $\mathcal{E}_{N,0}(D)\neq0$.\\
We suppose the converse:  $\mathcal{E}_{N,0}(D)=0$ and we show that it is absurd.\\
If $\mathcal{E}_{N,0}(D)=0$ then the columns of the matrix are linearly dependent. Hence there are rational functions $c_i(X,Y) \in \QQ(X,Y)$ such that 
\begin{equation}\label{Nj}
\mathcal{N}_j:=\sum_{i=1}^kc_iD^j(v_i)=0, \, j=0,\dots,k-1
\end{equation}
with $k=(N+1)(N+2)/2-1$. Now, take $k$ to be the smallest value such that for $i=1,\dots,k$  there exists rational functions $c_i$,  not all zero, and $v_i \in \CC[X,Y]_{\leq N,0}$, linearly independent over $\QQ$ such that equalities (\ref{Nj}) holds. It is clear that $k>1$.\\
We have:
$$D(\mathcal{N}_j)-\mathcal{N}_{j+1}=\sum_{i=1}^kD(c_i)D^j(v_i)=0, \, j=0,\dots,k-2,$$
and so from the minimality of $k$, we see that the terms $D(c_i)$ must all vanish. Hence, each of the $c_i$ are either rational first integrals or constants.\\
Now, we consider the polynomial  $G=\lambda p + \mu q$ where $(\lambda:\mu)$ does not belongs to $\sigma(p,q) \cup \{(\lambda_0:\mu_0)\}$.  This choice is possible because by Proposition \ref{bornespectre}, $\sigma(p,q)$ is finite. Then $G$ is absolutely irreducible, $\deg G=N$ and $G(0,0)\neq 0$.  We also remark that $c_i$ are constants for all $(x,y)$ such that $G(x,y)=0$, because $p/q$ is a first integral. We denote by $c_i(\lambda,\mu)$ these constants. This gives
$$\mathcal{N}_0^{\lambda,\mu}:=\sum_{i=1}^k c_i(\lambda,\mu) v_i \neq 0, \textrm{ and } G=\lambda p + \mu q \textrm{ divides } \mathcal{N}_0^{\lambda,\mu}.$$
Indeed, $v_i$ are linearly independent,  and furthermore if $G(x,y)=0$ then $\mathcal{N}_0^{\lambda,\mu}(x,y)=\mathcal{N}_0(x,y)=0$.
Thus, there exists $c\in \QQ$ such that $c.G=\mathcal{N}_0^{\lambda,\mu}$. As by construction, $\mathcal{N}_0^{\lambda,\mu} \in \CC[X,Y]_{\leq N,0}$, we deduce that $G$ belongs to $\CC[X,Y]_{\leq N,0}$ and this contradicts  $G(0,0) \neq 0$.  Thus we obtain $\mathcal{E}_{N,0}(D) \neq 0$.\\

Second, we prove that $F=\lambda_0 p + \mu_0 q$  divides $\mathcal{E}_{N,0}(D)$.\\
$F$ is a Darboux polynomial and belongs to $\CC[X,Y]_{\leq N,0}$.
Since $\mathcal{E}_{N,0}(D)$ is independent of the chosen basis (with the same arguments used in Proposition \ref{chg_base}), we can choose a basis where $v_1=F$. Then as in Proposition \ref{Pereira_Prop} we deduce that
 $F$ is a factor of $\mathcal{E}_{N,0}(D)$ and this concludes the proof.
\end{proof}

\subsection{Computation of a rational first integral}
Thanks to  Proposition \ref{Prop_ecst_moins_zero}, we can describe an algorithm which computes a rational first integral.\\

We denote by $(x_k,y_k)$, $k=1,\dots,N^6$, the points in $S\times S$ where\\ \mbox{$S=\{0,\dots, N^3-1\}$} and by $D_k$ the following derivation:
 $$D_k=A(X+x_k,Y+y_k)\partial_X+B(X+x_k,Y+y_k)\partial_Y.$$
Let $g \in \QQ[X,Y]$ with degree smaller than $d$, $\mathcal{L}_{k,g}$ is the following linear map:
\begin{eqnarray*}
\mathcal{L}_{k,g}:\QQ[X,Y]_{\leq N}& \longrightarrow& \QQ[X,Y]_{\leq N+d-1}\\
f&\longmapsto& D_k(f)-g.f
\end{eqnarray*}

\begin{Lem}\label{lem_dimker2}
Suppose that $D_k$ has a rational first integral $p/q$ of degree $N$. We denote by $g$ the cofactor of $p$ and $q$. \\
Then $\dim_{\QQ} \ker \mathcal{L}_{k,g}=2$ and if we denote by $\{\tilde{p},\tilde{q}\}$ a basis of $\ker \mathcal{L}_{k,g}$ then $\tilde{p}/\tilde{q}$ is a rational first integral of $D_k$.
\end{Lem}

\begin{proof}
This proof follows very closely the first part of the proof of Theorem 6 in \cite{Ollagnier}.\\
If $f \in \ker \mathcal{L}_{k,g}$ then by Darboux's theorem, see Theorem  \ref{darbouxthm}, there exist $\alpha$ and $\beta$ in $\QQ$ such that $f=\alpha p + \beta q$ or $f$ has degree less than $N$ and divides $\alpha p + \beta q$. The last case implies that $(\alpha,\beta) \in \sigma(p,q)$. We deduce that:
\begin{equation}\label{star}
\ker \mathcal{L}_{k,g}=Span(p,q) \cup Span (f_1)\cup \dots \cup Span(f_m),
\end{equation}
where $f_i$ is a factor of $\alpha p+ \beta q$, $(\alpha, \beta) \in \sigma(p,q)$, and the cofactor of $f_i$ is equal to $g$. Remark that this union is finite because $\sigma(p,q)$ is finite.\\
As $\QQ$ is infinite, equality (\ref{star}) implies $\ker \mathcal{L}_{k,g}=Span(p,q)$.
\end{proof}

\textbf{\textsf{Algorithm Rat-First-Int}}\\
\textsf{Input:} A polynomial derivation $D=A(X,Y)\partial_X+B(X,Y)\partial_Y$, and $N$ an integer.\\
\textsf{Output:} A rational first integral with degree smaller than $N$ or ``There exists no rational first integral with degree smaller than $N$''.
\begin{enumerate}
\item \label{rfi1}Compute $\mathcal{E}_N(D)$.
\item \label{rfi2}If $\mathcal{E}_N(D)\neq 0$ then Return  ``There exists no rational first integral with degree smaller than $N$'', else go to step \ref{rfi3}, end If.
\item \label{rfi3} Compute the smallest integer $n$ such that $\mathcal{E}_n(D)=0$ and $\mathcal{E}_{n-1}(D) \neq 0$.
\item \label{rfi5}  Set F:=0; k:=1;
\item \label{while} While $F=0$ do 
\begin{enumerate}
\item Compute $\mathcal{E}_{n,0}(D_k)$.
\item Compute all the irreducible factors $f_1,\dots,f_m$ of $\mathcal{E}_{n,0}(D_k)$ with degree equal to $n$.
\item For all i:=1, \dots, m do:  If $\gcd\big(f_i,D_k(f_i)\big)=f_i$ then set $F:=f_i$ and go to step \ref{calculcof}, end If; end For.
\item $k:=k+1$;\\
end While.
\end{enumerate}
\item \label{calculcof} Compute the cofactor $g:=D_{k-1}(F)/F$.
\item Compute a basis $\{p,q\}$  of $\ker \mathcal{L}_{k-1,g}$.
\item Return $p/q(X-x_{k-1},Y-y_{k-1})$.
\end{enumerate}

\begin{Prop}\label{Prop-rat-first-int-cor}
The algorithm \textsf{Rat-First-Int}  terminates and uses the While loop  at most $N^6$ times. Furthermore the algorithm \textsf{Rat-First-Int} is correct.
\end{Prop}

\begin{proof}
\emph{The algorithm terminates.}  We just have to show that the While loop terminates.\\
In Step \ref{while}, thanks to Proposition \ref{ecstatic_nulle}, $D$ has a reduced rational first integral $p/q$ of degree $n \leq N$. Then $D_k$ has a rational first integral $p_k/q_k(X,Y)=p/q(X+x_k,Y+y_k)$ of degree $n \leq N$.\\
Furthermore $p/q$ is non-composite (thus $p_k/q_k$ is  also non-composite). Indeed, if $p/q$ is composite then $p/q=u \circ h$ with $\deg h < n$. This gives
$$0=D(p/q)=D(u \circ h)= u'(h)D(h),$$
and $u'(h) \neq 0$ because $\deg u \geq 2$. Therefore $D(h)=0$ and $h$ is a rational first integral with degree smaller than $n$. Thus by Proposition \ref{ecstatic_nulle},  $\mathcal{E}_{\deg h}(D) =0$, this contradicts the minimality of $n$.  We deduce then: $p/q$ is non-composite.

Now, remark that the algorithm terminates if  the following two conditions are satisfied:
\begin{equation}\label{cond1}
p(x_k,y_k)\neq 0 \textrm{ or }q(x_k,y_k)\neq 0,
\end{equation}
 and 
\begin{equation}\label{cond2}
\big( -q(x_k,y_k): p(x_k,y_k) \big) \not \in \sigma(p_k,q_k)=\sigma(p,q).
\end{equation}
Indeed, in this situation we can apply Proposition \ref{Prop_ecst_moins_zero} and we deduce that there exists a polynomial $F=-q(x_k,y_k)p_k(X,Y)+p(x_k,y_k)q_k(X,Y)$ absolutely irreducible such that $gcd\big(F,D_k(F)\big)=F$. \\

Now we show that there exists a point $(x_k,y_k)$ in $\{0,\dots, N^3-1\}^2$  such that (\ref{cond1}) and (\ref{cond2}) are satisfied.\\
By Bezout's Theorem we just have to avoid $N^2$ points to satisfy (\ref{cond1}).\\
Now, we consider the polynomial
$$\mathcal{P}(X,Y)=\prod_{(\lambda:\mu) \in \sigma(p,q)} \big(\lambda p(X,Y) +\mu q(X,Y) \big).$$
We remark that if $\mathcal{P}(x_k,y_k) \neq 0$ then (\ref{cond2}) is satisfied. Furthermore, by Proposition \ref{bornespectre}, $\deg \mathcal{P} \leq N(N^2-1)$. 
Zippel-Schartz's lemma, see \cite[Lemma 6.44]{GG}, implies that $\mathcal{P}$ has at most $N^6-N^5$ roots in  $\{0,\dots, N^3-1\}^2$. Then the algorithm terminates and uses the While loop at most $N^6-N^5+N^2+1$ times.\\

\emph{The algorithm is correct.}
If $f_i$ satisfies $\gcd\big(f_i,D_k(f_i)\big)=f_i$ then $f_i$ is a Darboux polynomial. Then by Darboux's theorem, see Proposition \ref{darbouxthm}, we deduce  that $f_i =\alpha p_k + \beta q_k$ because $\deg p_k/q_k=n=\deg f_i$. \\
As $p_k/q_k$ is a first integral, the cofactors of $p_k$ and $q_k$ are equal. Then, we deduce that the cofactors of $p_k$, $q_k$ and $f_i$ are equal. This cofactor is the polynomial $g$. Then thanks to Lemma \ref{lem_dimker2} we  conclude that the algorithm is correct.
\end{proof}

Now we can prove Theorem \ref{Thm2}.
\begin{proof}
Thanks to Proposition \ref{Prop-rat-first-int-cor}, we just have to prove that the  algorithm \textsf{Rat-First-Int} works with $\bigO\Big( \big(dN\log(\mathcal{H})\big)^{\bigO(1)}\Big)$ binary operations.\\
We have already mention that we can  compute determinants and solve linear systems in polynomial-time. Thus we can perform Step \ref{rfi1}, Step \ref{rfi2} and Step \ref{rfi3}  with  $\bigO\Big( \big(dN\log(\mathcal{H})\big)^{\bigO(1)}\Big)$ binary operations.\\
Now, we study the  While loop.\\
 We recall here that  we can shift the variable of a polynomial in polynomial-time see e.g. \cite[Problem 2.6]{BiPa1994}. We can also perform linear algebra, compute gcd and divide polynomials in polynomial-time, see e.g. \cite{GG}. Furthermore,  as  in Corollary \ref{Cor-size} we can show that the bit-size of $\mathcal{E}_{N,0}(D)$ belongs to
 $ \bigO  \Big( \big(dN\log(\mathcal{H})\big)^{\bigO(1)}\Big)$.  Then we can factorize $\mathcal{E}_{N,0}(D)$ with  $\bigO\Big( \big(dN\log(\mathcal{H})\big)^{\bigO(1)}\Big)$ binary operations. As we use the While loop at most $N^6$ times we obtain the desired result.
\end{proof}

\section{Open questions}\label{Sect-openquestions}
\subsection{Liouvillian first integrals}
In this paper we have shown how to compute efficiently Darboux polynomials. Thus this improves the complexity of  Prelle-Singer's method. Now the question is: Can we compute efficiently an integrating factor corresponding to a Liouvillian first integral?\\
In \cite{Duarte1,Duarte2,Duarte3} the authors give an algorithm to compute such integrating factors. The key point is the computation of exponential factors. The definition of an exponential factor is the following:
\begin{Def}
Given $f,g \in \CC[X,Y]$, we say that $e=\exp(g/f)$ is an exponential factor of the derivation $D$ if $D(e)/e$ is a polynomial  of degree at most $d-1$. 
\end{Def}

In \cite{Christopher_Llibre_Pereira} the authors define the integrable multiplicity and the algebraic multiplicity.
\begin{Def}
We say that a Darboux polynomial $f$ has \emph{integrable multiplicity} $m$ with respect to a derivation $D$, if $m$ is the largest integer for which the following is true: there are $m-1$ exponential factors $\exp(g_j/f^j)$,  $j=1,\dots,m-1$, with $\deg g_j \leq j .\deg f$, such that each $g_j$ is not a multiple of $f$.\\
We say that a Darboux polynomial $f$ of degree $N$ has \emph{algebraic multiplicity} $m$ with respect to a derivation $D$, if $m$ is the greatest positive integer such that the $m$th power of $f$ divides $\mathcal{E}_N(D)$.
\end{Def}

C. Christopher, J. LLibre and J.V. Pereira in \cite{Christopher_Llibre_Pereira} show that these multiplicities are equal when we consider absolutely irreducible Darboux polynomials. This is a deep result, but unfortunately nowadays there is no simple characterization of exponential factor $\exp(g/f)$ when $f$ is reducible. If we can characterize exponential factors with the ecstatic curve then perhaps we will compute efficiently an integrating factor corresponding to a Liouvillian first integral.\\


\subsection{Inverse integrating factor}
An inverse integrating factor is a Darboux polynomial with cofactor $\dv(A,B)$.
An inverse integrating factor $R$ has the following interesting property: The algebraic limit cycles of the polynomial vector field corresponding to $D$ are factors of $R$, see e.g. \cite{Giac_Llibre_Viano}. For other results we can read e.g. \cite{chavarriga_giac_gine_llibre}. We remark easily that:
\begin{eqnarray}\label{epanadiplose}
R \textrm{ is an inverse integrating factor} & \iff & \partial_X\Big( \dfrac{A}{R} \Big)= \partial_Y\Big( \dfrac{B}{R} \Big) \nonumber\\
& \iff & A \partial_X R +B \partial_Y R = \dv(A,B)R. \quad \quad \,
\end{eqnarray}

For a given integer $N$ we can compute, if it exists, $R$ with $\deg R \leq N$. Indeed, we just have to solve the linear system (\ref{epanadiplose}). With this strategy  and with classical tools of linear algebra we can compute $R$ with  $\bigO(N^6)$ arithmetic operations if $N \geq d$.\\

This kind of linear system also appears when we study absolute factorization, see  \cite{Ruppert,Ruppert2,Gao,CL,Scheib}. Indeed, we can compute the absolute factorization of a given polynomial $R$  with  a solution $(A,B)$ of (\ref{epanadiplose}). 


In \cite{CL} the authors use this kind of linear system and show that the absolute factorization of $R$ can be performed with $\bigOsoft(N^4)$ arithmetic operations. We recall that ``soft Oh'' is  used for
readability in order to hide logarithmic factors in cost estimates. Then the question is the following:\\
Can we perform the computation of an inverse integrating factor in a deterministic way  with $\bigOsoft(N^4)$ arithmetic operations instead of $\bigO(N^6)$?
\section{Acknowledgment}
I thank L.~Bus\'e, G.~Lecerf, J.~Moulin-Ollagnier, and J.-A.~Weil for their encouragements during the preparation of this work.



\newcommand{\etalchar}[1]{$^{#1}$}

\end{document}